\documentclass[]{interact}

\usepackage{mathtools,mathrsfs}
\usepackage{enumitem}
\usepackage{xcolor}
\usepackage[hidelinks]{hyperref}
\usepackage[nameinlink,capitalize,noabbrev]{cleveref}
\usepackage[longnamesfirst,sort]{natbib}
\bibpunct[, ]{(}{)}{;}{a}{,}{,}
\renewcommand\bibfont{\fontsize{9}{10.5}\selectfont}
\hypersetup{
  pdftitle={Morita Transport, Marked Projective Loci, and Dedekind
            Classification for C4-Type Conditions},
  pdfauthor={Chandrasekhar Gokavarapu, Sekhar Babu Gosala, Sudhakar Gadde,
             Rajasekhar Yedla, and Durga Ratna Sai Ryali},
  pdfsubject={C4-type module conditions, Morita transport, and Dedekind
              module classification},
  pdfkeywords={C4-module, C4-star module, Morita equivalence, projective
               monoid, Dedekind domain, torsion submodule}
}

\newtheorem{theorem}{Theorem}[section]
\newtheorem{proposition}[theorem]{Proposition}
\newtheorem{lemma}[theorem]{Lemma}
\newtheorem{corollary}[theorem]{Corollary}
\theoremstyle{definition}
\newtheorem{definition}[theorem]{Definition}

\theoremstyle{remark}
\newtheorem{remark}[theorem]{Remark}

\newcommand{\Mod}{\operatorname{Mod}}
\newcommand{\End}{\operatorname{End}}
\newcommand{\Hom}{\operatorname{Hom}}
\newcommand{\im}{\operatorname{Im}}
\newcommand{\soc}{\operatorname{soc}}

\newcommand{\Vmon}{\mathrm V}
\newcommand{\Cfour}{\mathrm{C4}}
\newcommand{\Cfourstar}{\mathrm{C4}^{\ast}}
\newcommand{\leess}{\leq_{\mathrm e}}
\newcommand{\lesum}{\leq_{\oplus}}

\begin{document}

\articletype{RESEARCH ARTICLE}

\title{Morita Transport, Marked Projective Loci, and Dedekind
Classification for $\Cfour$-Type Conditions}

\author{
\name{Chandrasekhar Gokavarapu\textsuperscript{a,b}\thanks{CONTACT
Chandrasekhar Gokavarapu. Email:
\href{mailto:chandrasekhargokavarapu@gmail.com}{chandrasekhargokavarapu@gmail.com}.
ORCID:
\href{https://orcid.org/0009-0006-5306-371X}{0009-0006-5306-371X}},
Sekhar Babu Gosala\textsuperscript{a,c},
Sudhakar Gadde\textsuperscript{a,d},
Rajasekhar Yedla\textsuperscript{a,d}, and
Durga Ratna Sai Ryali\textsuperscript{e}}
\affil{\textsuperscript{a}Lecturer in Mathematics, Department of
Mathematics, Government College (Autonomous), Rajahmundry, Andhra Pradesh
533105, India;
\textsuperscript{b}Research Scholar, Department of Mathematics, Acharya
Nagarjuna University, Guntur, Andhra Pradesh 522510, India;
\textsuperscript{c}Research Scholar, Department of Engineering Mathematics,
Andhra University, Visakhapatnam, Andhra Pradesh 530003, India;
\textsuperscript{d}Research Scholar, Department of Mathematics, Godavari
Global University, Rajamahendravaram, Andhra Pradesh 533296, India;
\textsuperscript{e}Student, M.Sc. Mathematics Programme, 2025--27 Batch,
Department of Mathematics, Government College (Autonomous), Rajahmundry,
Andhra Pradesh 533105, India}
}

\maketitle

\begin{abstract}
Objectwise Morita invariance of $\Cfour$ is known.  We prove the corresponding
statements for $\Cfourstar$, semi-weak-CS, and strongly $\Cfourstar$,
including the chain joins in the definition of semi-weak-CS.  For a module
property $\mathcal Q$ preserved and reflected by equivalences, we consider
the subset $\Vmon_{\mathcal Q}(R)$ of the finitely generated projective
monoid $\Vmon(R)$.  This marked pair is transported by Morita equivalence.
The associated regular-module ring property is Morita invariant if and only
if membership in $\Vmon_{\mathcal Q}(R)$ is constant on the order units.
The matrix profile satisfies
$\mu_{\mathcal Q}(M_m(R))=\{n\geq1:mn\in\mu_{\mathcal Q}(R)\}$.
For a nondivision right Ore domain, the profiles of $\Cfour$, $\Cfourstar$,
and strongly $\Cfourstar$ are $\{1\}$, whereas the semi-weak-CS profile is
$\mathbb N_{\geq1}$.  Over a nonfield Dedekind domain $D$, the first three
marked loci consist of the projectives of rank at most one, while every
finitely generated projective is semi-weak-CS.  For a finitely generated
$D$-module $M$, we prove that every $\Cfour$ defect at rank at most one can
be transferred between $M$ and $\tau(M)$, and we construct a defect when the
rank is at least two.  Hence $M$ is $\Cfour$ if and only if
$\operatorname{rank}_D M\leq1$ and $\tau(M)$ is $\Cfour$; the analogous
criterion holds for $\Cfourstar$.  Combined with the finite-length torsion
criterion, these reductions give invariant-factor tests for the four
conditions and criteria for their behaviour under direct sums.
\end{abstract}

\begin{keywords}
$\Cfour$-module; Dedekind domain; Morita equivalence; order unit; projective
monoid; torsion radical
\end{keywords}

\begin{amscode}
2020 Mathematics Subject Classification: 16D90; 16D40; 16S50; 13C05; 13C10;
13F05; 18E35
\end{amscode}

\section{Introduction}

\subsection*{Existing literature}

The $\Cfour$ condition was introduced as a common generalization of the
$\mathrm C3$ and square-free conditions \citep{DingEtAl2017}.  That work gives characterizations, decomposition
results, and consequences for pseudo-continuous modules; it also observes
that the $\Cfour$ property of an individual module is preserved under Morita
equivalence.  Subsequent work connected $\Cfour$ and its dual version to
perspective direct summands \citep{AltunEtAl2018}, examined exchange phenomena
\citep{IbrahimYousif2022a,IbrahimYousif2022b}, and developed related uniqueness
conditions for direct complements \citep{IbrahimYousif2023,KeskinEtAl2025}.
Further variations around the $\mathrm C3$--$\Cfour$ hierarchy appear in
\citep{AminEtAl2015,Zhu2023}.

The next stage of the theory concerns modules all of whose submodules are
$\Cfour$.  \citet{IbrahimEidElGuindy2026} call these modules $\Cfourstar$ and
combine $\Cfourstar$ with semi-weak-CS to define strongly $\Cfourstar$
modules.  Their results establish structural and
decomposition properties for these newer classes.  Semi-weak-CS is the
non-finitary part of the definition: it tests the unions arising from chains
of compatible isomorphisms between disjoint semisimple direct summands.
Its transport therefore uses preservation of chain joins in addition to the
finite kernel--image argument used for $\Cfour$.

The standard theory of Morita equivalence explains why direct sums,
subobjects up to isomorphism, exact sequences, and projective generators are
categorical data \citep{AndersonFuller1992,Lam1999}.  Equivalence of the
finitely generated projective subcategories also yields the familiar Morita
invariance of the projective monoid $\Vmon(R)$ \citep{Goodearl2009}.  Morita
theory nevertheless shows that the
regular module is not categorical as a distinguished object: an equivalence
sends it to a finitely generated progenerator, which need not be the regular
module on the other side.

An earlier preprint of the first author stated regular-module Morita
invariance for these conditions
\citep{Gokavarapu2026MoritaPreprint}.  The present article replaces that
statement by an order-unit criterion.  The distinction is that an equivalence
sends $R_R$ to a finitely generated progenerator, which need not correspond
to $S_S$.  The right Ore-domain calculation below gives examples in which
the regular-module property changes within a Morita class.

Two classical structure theories become relevant once this distinction is
made.  A right Ore domain is uniform as a right module, so intersections of
nonzero right ideals provide a noncommutative replacement for the elementary
ideal argument over a commutative domain; see
\citep{GoodearlWarfield2004,McConnellRobson2001}.  At the projective level, the
Steinitz theorem identifies a nonzero finitely generated projective module
over a Dedekind domain with $D^{r-1}\oplus I$, where $I$ is invertible.  Thus
its isomorphism class is determined by its rank and a Picard-class coordinate
\citep{AtiyahMacdonald1969,Weibel2013}.  These facts allow the marked locus to
be calculated outside the free ray detected by matrix rings.

For the torsion side, a separate finite-length classification is available:
over a Dedekind domain, a primary torsion module is $\Cfour$ exactly in the
homocyclic case, and it is $\Cfourstar$ exactly when it is cyclic or
semisimple \citep[Theorem~6.4]{Gokavarapu2026SmithOrbit}.  That result concerns
finite-length torsion modules.  It does not determine how a torsion part
interacts with a nonzero torsion-free part inside a general finitely generated
module.  The mixed problem requires control of arbitrary decompositions and
of maps crossing the torsion--torsion-free boundary.

\subsection*{Problem and scope}

The cited literature supplies the objectwise $\Cfour$ observation and
structural results for the newer conditions.  Objectwise transport, however,
does not compare the two regular modules of Morita-equivalent rings, because
the image of a regular module is generally a progenerator.  It also does not
determine which nonfree finitely generated projectives have the property.
For finitely generated modules over a Dedekind domain, the projective and
finite-torsion cases leave a further question: how the $\Cfour$ tests behave
when the torsion-free and torsion parts are both nonzero.  Accordingly, this
article addresses the following four problems.

\begin{enumerate}[label=(\roman*)]
\item If $F:\Mod\text{-}R\to\Mod\text{-}S$ is an equivalence and $M$ is an
      individual module, do $\Cfourstar$, semi-weak-CS, and strongly
      $\Cfourstar$ pass from $M$ to $F(M)$?
\item Give a criterion for when the regular-module property is constant on a
      Morita class.
\item Determine the projective modules carrying the four conditions in
      classes where the projective monoid can be computed.
\item Relate $\Cfour$ defects of a finitely generated Dedekind module
      $M\cong P\oplus\tau(M)$ to defects of its torsion radical.
\end{enumerate}

\subsection*{Results proved in this article}

The main results are as follows.

\begin{enumerate}[label=(\arabic*)]
\item \Cref{thm:objectwise} proves objectwise transport for $\Cfourstar$,
      semi-weak-CS, and strongly $\Cfourstar$.  In the semi-weak-CS case the
      proof uses the joins of the compatible chains in the defining
      subobject lattice.
\item \Cref{thm:markedV,thm:ordercriterion} show that Morita equivalence
      transports $(\Vmon(R),\Vmon_{\mathcal Q}(R))$ and that the
      regular-module ring property is Morita invariant if and only if the
      marking is constant on order units.  The progenerator, matrix, and
      full-corner formulas follow from the same comparison.  The matrix
      profiles satisfy
      $\mu_{\mathcal Q}(M_m(R))=\{n:mn\in\mu_{\mathcal Q}(R)\}$.
\item \Cref{thm:oredichotomy} computes the four profiles for right Ore
      domains.  For a nondivision right Ore domain the $\Cfour$,
      $\Cfourstar$, and strongly $\Cfourstar$ profiles are $\{1\}$, while
      the semi-weak-CS profile is $\mathbb N_{\geq1}$.
\item \Cref{thm:dedekindlocus,thm:dedekindrigidity} compute the four marked
      loci for a nonfield Dedekind domain and determine the corresponding
      regular-module properties in its Morita class.
\item \Cref{thm:defect-reconstruction} transfers $\Cfour$ defects between a
      rank-at-most-one module and its torsion radical and constructs a defect
      in rank at least two.  This gives the $\Cfour$ and $\Cfourstar$
      torsion-radical criteria in
      \cref{thm:torsion-reduction-c4,thm:torsion-reduction-c4star}.
\item Combining those reductions with the finite-length torsion criterion
      from \citet[Theorem~6.4]{Gokavarapu2026SmithOrbit} gives the
      invariant-factor tests in
      \cref{thm:complete-dedekind-classification}.  The direct-sum criterion
      in \cref{thm:direct-sum-compatibility} records the additional
      compatibility required on common primary support.
\end{enumerate}

The categorical and Morita background is taken from
\citep{AndersonFuller1992,Lam1999,MacLane1998,Mitchell1965}; the language of
essentiality and extending modules may be found in
\citep{DungEtAl1994,ClarkEtAl2006,BirkenmeierParkRizvi2013}.

\section{The four module conditions}

All rings are associative with identity and all modules are unitary right
modules.  For a module $M$, the notation $X\lesum M$ means that $X$ is a direct
summand, while $X\leess M$ means that $X$ is essential in $M$.

\begin{definition}\label{def:c4}
A module $M$ is a \emph{$\Cfour$-module} if, whenever
$M=A\oplus B$ and $f:A\to B$ is a homomorphism with
$\ker f\lesum A$, one has $\im f\lesum B$.
It is a \emph{$\Cfourstar$-module} if every submodule of $M$ is a
$\Cfour$-module.
\end{definition}

This is the standard definition from the $\Cfour$ literature
\citep{DingEtAl2017,IbrahimEidElGuindy2026}.  We next state the chain condition
used in the definition of semi-weak-CS.

\begin{definition}\label{def:scriptS}
Let $M$ be a module.  Write $\mathcal{S}(M)$ for the set of triples
$(X,Y,\varphi)$ such that
\[
 X,Y\lesum M,\qquad X\cap Y=0,
\]
$X$ and $Y$ are semisimple, and $\varphi:X\to Y$ is an isomorphism.  Put
\[
 (X,Y,\varphi)\preccurlyeq (X',Y',\varphi')
\]
when $X\leq X'$, $Y\leq Y'$, and $\varphi'|_X=\varphi$.
For a chain $\mathcal C\subseteq\mathcal{S}(M)$, set
\[
 X_{\mathcal C}=\bigcup_{(X,Y,\varphi)\in\mathcal C}X,
 \qquad
 Y_{\mathcal C}=\bigcup_{(X,Y,\varphi)\in\mathcal C}Y.
\]
\end{definition}

The compatibility in the order gives an induced isomorphism
$\varphi_{\mathcal C}:X_{\mathcal C}\to Y_{\mathcal C}$.  The two unions are
semisimple, disjoint submodules; see the original construction in
\citep{IbrahimEidElGuindy2026}.

\begin{definition}\label{def:swcs}
A module $M$ is \emph{semi-weak-CS} if, for every chain
$\mathcal C\subseteq\mathcal{S}(M)$, there exist direct summands
$A,B\lesum M$ such that
\[
 X_{\mathcal C}\leess A,\qquad Y_{\mathcal C}\leess B.
\]
It is \emph{strongly $\Cfourstar$} if it is both $\Cfourstar$ and
semi-weak-CS.
\end{definition}

\begin{proposition}[The one-pair test is vacuous]\label{prop:pair}
If \cref{def:swcs} is replaced by the corresponding condition on one triple
$(X,Y,\varphi)\in\mathcal{S}(M)$, then every module satisfies the resulting
condition.
\end{proposition}

\begin{proof}
For a single triple, take $A=X$ and $B=Y$.  Then $A,B\lesum M$, while
$X\leess A$ and $Y\leess B$.  Hence a one-pair condition has no nontrivial
content; the force of \cref{def:swcs} lies in the unions of chains, which need
not be direct summands.
\end{proof}

\section{Categorical transport}

We record the categorical facts in the precise form needed below.  An
equivalence between module categories is additive and exact; it preserves and
reflects split monomorphisms, kernels, images, finite limits, and finite
colimits.  It also induces an order isomorphism between subobject lattices.
Because an order isomorphism between complete lattices preserves all joins, it
preserves unions of chains of subobjects.  These are standard consequences of
equivalence; see \citep{AndersonFuller1992,Lam1999,MacLane1998}.

\begin{lemma}\label{lem:essential}
Let $F:\Mod\text{-}R\to\Mod\text{-}S$ be an equivalence.  A monomorphism
$u:X\rightarrowtail Y$ is essential if and only if $F(u)$ is essential.
\end{lemma}

\begin{proof}
Use the categorical characterization: $u$ is essential when, for every
$h:Y\to Z$, monicity of $hu$ implies monicity of $h$.  Equivalences preserve
and reflect monomorphisms and composition, and a quasi-inverse gives the reverse
implication.
\end{proof}

\begin{lemma}\label{lem:semisimple}
An equivalence $F:\Mod\text{-}R\to\Mod\text{-}S$ preserves and reflects
semisimple objects, direct summands, zero intersections, and isomorphisms.
It sends a chain $\mathcal C\subseteq\mathcal{S}(M)$ to a chain in
$\mathcal{S}(F(M))$ and identifies $F(X_{\mathcal C})$ and
$F(Y_{\mathcal C})$ with the corresponding subobject joins.
\end{lemma}

\begin{proof}
A module is semisimple precisely when every subobject inclusion into it splits.
Direct summands are split subobjects, intersections are pullbacks, and the
chain unions are joins in the subobject lattice.  Each description is
preserved and reflected by an equivalence.
\end{proof}

The $\Cfour$ equivalence in the next result is the observation of
\citet{DingEtAl2017}.  We include its proof together with the three other
conditions.  For semi-weak-CS, the required categorical datum is the join of
each compatible chain of subobjects.

\begin{theorem}[Objectwise transport]\label{thm:objectwise}
Let $F:\Mod\text{-}R\to\Mod\text{-}S$ be an equivalence and let
$M\in\Mod\text{-}R$.  Then each of the following equivalences holds:
\begin{align*}
 M\text{ is }\Cfour
   &\Longleftrightarrow F(M)\text{ is }\Cfour,\\
 M\text{ is }\Cfourstar
   &\Longleftrightarrow F(M)\text{ is }\Cfourstar,\\
 M\text{ is semi-weak-CS}
   &\Longleftrightarrow F(M)\text{ is semi-weak-CS},\\
 M\text{ is strongly }\Cfourstar
   &\Longleftrightarrow F(M)\text{ is strongly }\Cfourstar.
\end{align*}
\end{theorem}

\begin{proof}
Suppose $M=A\oplus B$ and $f:A\to B$.  Exactness identifies
$F(\ker f)$ and $F(\im f)$ with the kernel and image of $F(f)$, while split
subobjects are preserved and reflected.  This proves the $\Cfour$ assertion.

For $\Cfourstar$, every subobject of $F(M)$ is, up to isomorphism, the image of
a subobject of $M$ under $F$; apply the first assertion and then a
quasi-inverse.

For semi-weak-CS, let $\mathcal D$ be a chain in
$\mathcal{S}(F(M))$.  Applying a quasi-inverse gives, up to isomorphism, a chain
$\mathcal C$ in $\mathcal{S}(M)$.  By \cref{lem:semisimple}, the two coordinate
unions correspond.  If $M$ is semi-weak-CS, these unions are essential in
direct summands of $M$.  The images of those summands under $F$ are direct
summands of $F(M)$, and essentiality follows from \cref{lem:essential}.  The
converse is symmetric.  The strong assertion is now the conjunction of the
$\Cfourstar$ and semi-weak-CS assertions.
\end{proof}

\begin{remark}\label{rem:notfinite}
The $\Cfour$ condition has finite witnesses, but semi-weak-CS need not.  Its
definition permits arbitrary chains, so its obstruction theory is generally
not finite.  Any reduction from chain obstructions to single pairs or to a
least finite composition length requires an additional chain condition and
cannot be asserted for arbitrary modules.
\end{remark}

\section{The marked projective monoid}

Let $\mathcal Q$ be an isomorphism-invariant module property preserved and
reflected by equivalences of module categories.  In the applications,
$\mathcal Q$ is one of the four properties in \cref{thm:objectwise}.  We now
record not only whether the regular module has $\mathcal Q$, but which finitely
generated projective modules have it.

Recall that $\Vmon(R)$ is the abelian monoid of isomorphism classes $[P]$ of
finitely generated projective right $R$-modules, with
$[P]+[Q]=[P\oplus Q]$ \citep{Goodearl2009}.  Its algebraic preorder is
\[
 x\leq y\quad\Longleftrightarrow\quad
 y=x+z\text{ for some }z\in\Vmon(R).
\]
An element $u\in\Vmon(R)$ is an \emph{order unit} if, for every
$x\in\Vmon(R)$, one has $x\leq nu$ for some $n\geq1$.

\begin{definition}\label{def:markedV}
The \emph{$\mathcal Q$-marked projective locus} of $R$ is
\[
 \Vmon_{\mathcal Q}(R)
   =\{[P]\in\Vmon(R):P_R\text{ has }\mathcal Q\}.
\]
No closure of this subset under the monoid operation is assumed.
\end{definition}

\begin{theorem}[Marked projective-monoid transport]\label{thm:markedV}
Let $F:\Mod\text{-}R\to\Mod\text{-}S$ be an equivalence.  Then $F$ induces
an isomorphism of abelian monoids
\[
 \Vmon(F):\Vmon(R)\longrightarrow\Vmon(S),\qquad
 [P]\longmapsto[F(P)],
\]
and
\[
 \Vmon(F)\bigl(\Vmon_{\mathcal Q}(R)\bigr)
   =\Vmon_{\mathcal Q}(S).
\]
Consequently, the marked pair
$(\Vmon(R),\Vmon_{\mathcal Q}(R))$ is a Morita invariant.
\end{theorem}

\begin{proof}
An equivalence preserves and reflects projective objects and finite direct
sums.  It also preserves finite generation: equivalently, it restricts to an
equivalence between the finitely generated projective subcategories.  Thus it
induces the displayed monoid isomorphism, with the isomorphism induced by a
quasi-inverse as inverse.  Since $\mathcal Q$ is preserved and reflected,
$P$ has $\mathcal Q$ exactly when $F(P)$ has $\mathcal Q$, proving the marked
subset assertion.
\end{proof}

\begin{proposition}[Order units and progenerators]\label{prop:orderunit}
For a finitely generated projective module $P_{R}$, the following are
equivalent:
\begin{enumerate}[label=(\roman*)]
\item $[P]$ is an order unit of $\Vmon(R)$;
\item $R_R$ is a direct summand of $(P_{R})^n$ for some $n\geq1$;
\item $P_{R}$ is a progenerator.
\end{enumerate}
\end{proposition}

\begin{proof}
If $[P]$ is an order unit, then $[R]\leq n[P]$ for some $n$, which is exactly
(ii).  Condition (ii) says that the trace of $P$ is all of $R$, so the already
finitely generated projective module $P$ is a generator.  Conversely, a
progenerator satisfies (ii).  If (ii) holds and $X_R$ is finitely generated
projective, then $X$ is a direct summand of $R^m$ for some $m$ and hence of
$(P_{R})^{nm}$.  Therefore $[X]\leq nm[P]$, so $[P]$ is an order unit.
\end{proof}

Thus $[R]$ is one order unit of $\Vmon(R)$, but not a Morita-canonical one.
The image of a regular module under a Morita equivalence may represent a
different order unit.  This is the precise location of the regular-object
obstruction.

\section{Regular modules, progenerators, and order units}

We now compare the regular modules associated with a Morita equivalence.  Let
$\mathcal Q$ denote any module property invariant under equivalences of
module categories.  Each of the four properties in
\cref{thm:objectwise} may be used for $\mathcal Q$.

\begin{definition}
A ring $R$ is a \emph{right $\mathcal Q$-ring} when the regular module $R_R$
has $\mathcal Q$.
\end{definition}

\begin{theorem}[Progenerator formula]\label{thm:progenerator}
Let $P_R$ be a progenerator, put $S=\End_R(P)$, and use the standard
equivalence
\[
 F=\Hom_R(P,-):\Mod\text{-}R\longrightarrow\Mod\text{-}S
\]
with quasi-inverse $G=-\otimes_S P$.  Then
\begin{enumerate}[label=(\alph*)]
\item $S_S$ has $\mathcal Q$ if and only if $P_R$ has $\mathcal Q$;
\item $R_R$ has $\mathcal Q$ if and only if
      $\Hom_R(P,R)_S$ has $\mathcal Q$.
\end{enumerate}
\end{theorem}

\begin{proof}
There are natural isomorphisms $G(S_S)\cong P_R$ and
$F(R_R)=\Hom_R(P,R)_S$.  Apply equivalence invariance of $\mathcal Q$.
\end{proof}

The theorem shows that a comparison of $R_R$ and $S_S$ needs information about
the two different objects $R_R$ and $P_R$.  Morita equivalence alone supplies
no isomorphism between them.

\begin{corollary}[Matrix formula]\label{cor:matrix}
For every ring $R$ and $n\geq1$,
\[
 M_n(R)_{M_n(R)}\text{ has }\mathcal Q
 \quad\Longleftrightarrow\quad
 R_R^n\text{ has }\mathcal Q.
\]
\end{corollary}

\begin{proof}
Take $P_R=R_R^n$, for which $\End_R(P)\cong M_n(R)$, in
\cref{thm:progenerator}(a).
\end{proof}

\begin{corollary}[Full-corner formula]\label{cor:corner}
Let $e\in R$ be a full idempotent.  Then, for $S=eRe$,
\[
 S_S\text{ has }\mathcal Q
 \quad\Longleftrightarrow\quad
 eR_R\text{ has }\mathcal Q,
\]
and
\[
 R_R\text{ has }\mathcal Q
 \quad\Longleftrightarrow\quad
 Re_S\text{ has }\mathcal Q.
\]
\end{corollary}

\begin{proof}
Use the progenerator $eR_R$, the isomorphism
$\End_R(eR)\cong eRe$, and
$\Hom_R(eR,R)\cong Re$ as right $eRe$-modules.
\end{proof}

The marked monoid gives a criterion that separates Morita invariance of
$\Vmon(R)$ from constancy of the marked subset on the order units.

\begin{theorem}[Order-unit criterion]\label{thm:ordercriterion}
The right $\mathcal Q$-ring property is Morita invariant if and only if, for
every ring $R$, membership in $\Vmon_{\mathcal Q}(R)$ is constant on the order
units of $\Vmon(R)$.  Equivalently, for every progenerator $P_R$,
\[
 R_R\text{ has }\mathcal Q
 \quad\Longleftrightarrow\quad
 P_R\text{ has }\mathcal Q.
\]
\end{theorem}

\begin{proof}
Assume first that the right $\mathcal Q$-ring property is Morita invariant.
Let $P_R$ be a progenerator and put $S=\End_R(P)$.  Then $R$ and $S$ are
Morita equivalent, while \cref{thm:progenerator} gives
\[
 S_S\text{ has }\mathcal Q
 \quad\Longleftrightarrow\quad
 P_R\text{ has }\mathcal Q.
\]
Morita invariance compares the left side with $R_R$, so $R_R$ and $P_R$ have
the same $\mathcal Q$-status.  By \cref{prop:orderunit}, this proves constancy
on all order units.

Conversely, let $R$ and $S$ be Morita equivalent and choose an equivalence
$G:\Mod\text{-}S\to\Mod\text{-}R$.  The module $P=G(S_S)$ is a progenerator,
so $[P]$ and $[R]$ are order units.  Constancy on order units and invariance of
$\mathcal Q$ on corresponding objects give
\[
 R_R\text{ has }\mathcal Q
 \Longleftrightarrow P_R\text{ has }\mathcal Q
 \Longleftrightarrow S_S\text{ has }\mathcal Q.
\]
Thus the ring property is Morita invariant.
\end{proof}

\begin{corollary}[Morita-class form]\label{cor:classcriterion}
For a fixed ring $R$, the following are equivalent:
\begin{enumerate}[label=(\roman*)]
\item every ring Morita equivalent to $R$ has the same right
      $\mathcal Q$-status as $R$;
\item membership in $\Vmon_{\mathcal Q}(R)$ is constant on the order units of
      $\Vmon(R)$.
\end{enumerate}
\end{corollary}

\begin{proof}
If (i) holds and $P_R$ is a progenerator, apply (i) to $\End_R(P)$ and then
use \cref{thm:progenerator}.  This compares the $\mathcal Q$-status of $P_R$
with that of $R_R$.  Conversely, if (ii) holds and $S$ is Morita equivalent to
$R$, the image in $\Mod\text{-}R$ of $S_S$ under a quasi-inverse is a
progenerator.  Apply (ii) and objectwise invariance of $\mathcal Q$.
\end{proof}

\begin{corollary}[Finite-sum and summand criterion]\label{thm:closurecriterion}
Assume that $\mathcal Q$ is invariant under equivalences and is closed under
finite direct sums and direct summands.  Then the right $\mathcal Q$-ring
property is Morita invariant.
\end{corollary}

\begin{proof}
Let $P_R$ be any progenerator.  It is a direct summand of $R_R^m$ for some
$m$, while $R_R$ is a direct summand of $P_R^k$ for some $k$.  The closure
assumptions therefore give
$R_R$ has $\mathcal Q$ if and only if $P_R$ has $\mathcal Q$.  Apply
\cref{thm:ordercriterion}.
\end{proof}

The $\Cfour$ condition is not closed under finite direct sums in general.  The
next sections record a profile law, determine the profiles over right Ore
domains, and compute the marked loci over Dedekind domains.

\section{Matrix profiles and their scaling law}

\begin{definition}\label{def:profile}
For an equivalence-invariant module property $\mathcal Q$, define its
\emph{matrix profile} at $R$ by
\[
 \mu_{\mathcal Q}(R)=\{n\geq1:R_R^n\text{ has }\mathcal Q\}.
\]
\end{definition}

\begin{proposition}[The free ray in the marked monoid]\label{prop:profile}
For every $R$ and $n\geq1$,
\[
 n\in\mu_{\mathcal Q}(R)
 \quad\Longleftrightarrow\quad
 n[R]\in\Vmon_{\mathcal Q}(R)
 \quad\Longleftrightarrow\quad
 M_n(R)\text{ is a right }\mathcal Q\text{-ring}.
\]
\end{proposition}

\begin{proof}
The first equivalence is the identity $n[R]=[R^n]$ in $\Vmon(R)$; the second
is \cref{cor:matrix}.
\end{proof}

\begin{theorem}[Matrix-profile scaling]\label{thm:profilescaling}
For every ring $R$ and all $m\geq1$,
\[
 \mu_{\mathcal Q}(M_m(R))
   =\{n\geq1:mn\in\mu_{\mathcal Q}(R)\}.
\]
\end{theorem}

\begin{proof}
For $n\geq1$, \cref{prop:profile} and the canonical matrix-ring isomorphism
give
\begin{align*}
n\in\mu_{\mathcal Q}(M_m(R))
&\Longleftrightarrow
M_n(M_m(R))\text{ is a right }\mathcal Q\text{-ring}\\
&\Longleftrightarrow
M_{nm}(R)\text{ is a right }\mathcal Q\text{-ring}\\
&\Longleftrightarrow nm\in\mu_{\mathcal Q}(R).
\end{align*}
\end{proof}

\begin{corollary}\label{cor:profileconsequences}
For every $m\geq1$,
\[
 M_m(R)\text{ is a right }\mathcal Q\text{-ring}
 \quad\Longleftrightarrow\quad
 m\in\mu_{\mathcal Q}(R).
\]
Moreover, all matrix rings over $R$ are right $\mathcal Q$-rings if and only
if $\mu_{\mathcal Q}(R)=\mathbb N_{\geq1}$.
\end{corollary}

\begin{proof}
The first statement is \cref{prop:profile}.  The second follows by quantifying
the first over $m$.
\end{proof}

\section{The right Ore-domain dichotomy}

We next compute the profiles for right Ore domains.  This class contains all
commutative domains and many standard noncommutative domains.  A domain $R$ is called a
\emph{right Ore domain} when $aR\cap bR\neq0$ for all nonzero $a,b\in R$.
Equivalently, the nonzero elements satisfy the right Ore condition.  The
standard localization background may be found in
\citep{GoodearlWarfield2004,McConnellRobson2001}.

\begin{lemma}\label{lem:indecomp}
Every indecomposable module is a $\Cfour$-module.
\end{lemma}

\begin{proof}
If $M=A\oplus B$ and $M$ is indecomposable, then $A=0$ or $B=0$.  The image
condition in \cref{def:c4} is therefore automatic.
\end{proof}

\begin{lemma}\label{lem:zero-socle}
If $\soc(M)=0$, then $M$ is semi-weak-CS.
\end{lemma}

\begin{proof}
There are no nonzero semisimple submodules of $M$.  Hence every member of
$\mathcal{S}(M)$ is the zero triple and every chain has zero coordinate unions.
Both are essential in the zero direct summand.
\end{proof}

\begin{lemma}\label{lem:orestructure}
Let $R$ be a right Ore domain that is not a division ring.  Then:
\begin{enumerate}[label=(\alph*)]
\item $R_R$ is uniform, and every nonzero right ideal of $R$ is
      indecomposable;
\item $\soc(R^n)=0$ for every $n\geq1$.
\end{enumerate}
\end{lemma}

\begin{proof}
If $I$ and $J$ are nonzero right ideals, choose nonzero $a\in I$ and
$b\in J$.  The right Ore condition gives
$0\neq aR\cap bR\subseteq I\cap J$.  Hence $R_R$ is uniform.  Every nonzero
submodule of a uniform module is uniform and therefore indecomposable, proving
(a).

For (b), suppose that $0\neq S\leq R^n$ is simple and choose
$0\neq s=(s_1,\ldots,s_n)\in S$.  At least one coordinate is nonzero, and the
domain hypothesis gives $\operatorname{ann}_R(s)=0$.  Thus $sR\cong R_R$.
Since $sR=S$ is simple, $R_R$ would be simple.  Every nonzero $a\in R$ would
then satisfy $aR=R$ and hence be a unit, contrary to the assumption that $R$
is not a division ring.
\end{proof}

\begin{proposition}[The regular-element obstruction]\label{prop:oreobstruction}
Let $R$ be a right Ore domain that is not a division ring.  Then $R^n$ is not
a $\Cfour$-module for every $n\geq2$.
\end{proposition}

\begin{proof}
Choose a nonzero nonunit $a\in R$.  For $n\geq2$, write
\[
 R^n=A\oplus B,
 \qquad A=e_1R,
 \qquad B=e_2R\oplus\cdots\oplus e_nR,
\]
and define $f:A\to B$ by $f(e_1r)=e_2ar$.  Since $R$ is a domain,
$\ker f=0\lesum A$.  If $e_2aR$ were a direct summand of $B$, the quotient
\[
 B/e_2aR\cong R/aR\oplus R^{n-2}
\]
would be projective; hence its direct summand $R/aR$ would be projective.  The
sequence $0\to aR\to R\to R/aR\to0$ would split, making the nonzero right
ideal $aR$ a direct summand of the uniform module $R_R$.  Thus $aR=R$, so
$a$ has a right inverse.  Since $a$ is right regular, that inverse is
two-sided, contradicting the choice of $a$.  Therefore
$\operatorname{Im}f=e_2aR\not\lesum B$, and $R^n$ is not $\Cfour$.
\end{proof}

\begin{theorem}[Right Ore-domain profile dichotomy]\label{thm:oredichotomy}
Let $R$ be a right Ore domain.
\begin{enumerate}[label=(\roman*)]
\item If $R$ is a division ring, then
\[
 \mu_{\Cfour}(R)=\mu_{\Cfourstar}(R)
 =\mu_{\mathrm{semi\text{-}weak\text{-}CS}}(R)
 =\mu_{\mathrm{strong}\ \Cfourstar}(R)
 =\mathbb N_{\geq1}.
\]
\item If $R$ is not a division ring, then
\[
 \mu_{\Cfour}(R)=\mu_{\Cfourstar}(R)
 =\mu_{\mathrm{strong}\ \Cfourstar}(R)=\{1\},
 \qquad
 \mu_{\mathrm{semi\text{-}weak\text{-}CS}}(R)
 =\mathbb N_{\geq1}.
\]
\end{enumerate}
\end{theorem}

\begin{proof}
If $R$ is a division ring, every $R^n$ is semisimple.  Every submodule is
therefore a direct summand and a $\Cfour$-module, so $R^n$ is $\Cfourstar$.
For a chain in $\mathcal{S}(R^n)$, its two coordinate unions are subspaces and
hence direct summands; each is essential in itself.  Thus $R^n$ is also
semi-weak-CS and
strongly $\Cfourstar$.

Now suppose that $R$ is not a division ring.  Every nonzero right ideal of
$R$ is indecomposable by \cref{lem:orestructure}.  Together with
\cref{lem:indecomp}, this proves that $R_R$ is $\Cfourstar$.  The same lemma
gives $\soc(R^n)=0$ for every $n$, so \cref{lem:zero-socle} shows that all
$R^n$ are semi-weak-CS.  Hence $R_R$ is strongly $\Cfourstar$.  Finally,
\cref{prop:oreobstruction} shows that no $R^n$ with $n\geq2$ is $\Cfour$, and
therefore none is $\Cfourstar$ or strongly $\Cfourstar$.  The displayed
profiles follow.
\end{proof}

\begin{corollary}[Commutative and Noetherian specializations]
The conclusion of \cref{thm:oredichotomy} holds for every commutative domain
and for every right Noetherian domain.
\end{corollary}

\begin{proof}
Every commutative domain is right Ore.  A right Noetherian domain is right
Ore by the Goldie--Ore theorem; see \citep{GoodearlWarfield2004}.
\end{proof}

\begin{corollary}[An infinite counterexample family]\label{cor:orefamily}
Let $R$ be any right Ore domain that is not a division ring.  For every
$m\geq2$, the Morita-equivalent rings $R$ and $M_m(R)$ have different right
$\Cfour$, right $\Cfourstar$, and strongly right $\Cfourstar$ status.  Every
matrix ring $M_m(R)$ is, however, a right semi-weak-CS ring.
\end{corollary}

\begin{proof}
Combine \cref{prop:profile,thm:oredichotomy}.
\end{proof}

\begin{corollary}[Regular-module consequence]
Objectwise equivalence invariance of a module property does not, even for a
summand-theoretic property, imply Morita invariance of the associated
regular-module ring property.
\end{corollary}

\begin{proof}
Use any nondivision right Ore domain in
\cref{thm:objectwise,cor:orefamily}.
\end{proof}

\section{Dedekind marked loci and Morita rigidity}

The matrix profile records only the free ray
$\{n[D]:n\geq1\}$ in $\Vmon(D)$.  For a Dedekind domain the Steinitz theorem
makes it possible to calculate the marked subset on every finitely generated
projective.  We use the classical form: if $0\neq P_D$ is finitely generated
projective of rank $r$, then
\[
 P\cong D^{r-1}\oplus I
\]
for an invertible ideal $I$, whose class in $\operatorname{Pic}(D)$ is uniquely
determined by $P$.  Consequently,
\[
 \Vmon(D)\cong
 \{0\}\cup\bigl(\mathbb N_{\geq1}\times\operatorname{Pic}(D)\bigr),
\]
with addition given by addition of ranks and multiplication of ideal classes
\citep{AtiyahMacdonald1969,Weibel2013}.

\begin{lemma}[Rank-one projectives]\label{lem:rankone}
Let $D$ be a domain and let $P_D$ be a projective module of rank one.  Then
$P$ is $\Cfourstar$.
\end{lemma}

\begin{proof}
It is enough to show that every nonzero submodule $N\leq P$ is
indecomposable.  Let $K$ be the fraction field of $D$.  Since $P$ is
torsion-free of rank one, every nonzero submodule $N$ is torsion-free and
$K\otimes_DN$ has $K$-dimension one.  If $N=A\oplus B$ with $A,B\neq0$, then
torsion-freeness makes both $K\otimes_DA$ and $K\otimes_DB$ nonzero, giving a
decomposition of a one-dimensional $K$-space into two nonzero subspaces.  This
is impossible.  Thus every submodule of $P$ is indecomposable or zero, and
\cref{lem:indecomp} proves that every submodule is $\Cfour$.
\end{proof}

\begin{lemma}[Projectives have zero socle]\label{lem:dedekindsocle}
Let $D$ be a domain that is not a field.  Every finitely generated projective
module $P_D$ satisfies $\soc(P)=0$ and is therefore semi-weak-CS.
\end{lemma}

\begin{proof}
Some $D^m$ contains $P$ as a direct summand.  Since $D$ is a commutative Ore
domain, \cref{lem:orestructure}(b) gives $\soc(D^m)=0$.  Hence
$\soc(P)=0$, and \cref{lem:zero-socle} applies.
\end{proof}

\begin{theorem}[Dedekind marked-locus theorem]\label{thm:dedekindlocus}
Let $D$ be a Dedekind domain that is not a field.  Then
\begin{align*}
 \Vmon_{\Cfour}(D)
 &=\Vmon_{\Cfourstar}(D)
  =\Vmon_{\mathrm{strong}\ \Cfourstar}(D)\\
 &=\{[P]\in\Vmon(D):\operatorname{rank}_D P\leq1\},
\end{align*}
whereas
\[
 \Vmon_{\mathrm{semi\text{-}weak\text{-}CS}}(D)=\Vmon(D).
\]
Thus the three $\Cfour$-type marked loci contain every Picard class in rank
one and no class of rank at least two.
\end{theorem}

\begin{proof}
The zero module has all four properties.  If $P$ has rank one, then
\cref{lem:rankone} shows that $P$ is $\Cfourstar$, while
\cref{lem:dedekindsocle} shows that it is semi-weak-CS.  Hence $P$ is strongly
$\Cfourstar$.

Let now $P$ have rank $r\geq2$.  By the Steinitz theorem, write
\[
 P\cong D^{r-1}\oplus I=A\oplus B,
 \qquad A\cong D,
 \qquad B\cong D^{r-2}\oplus I,
\]
where $I$ is an invertible ideal.  Choose a nonzero nonunit $a\in D$ and
$0\neq x\in I$.  Then $0\neq axD\subsetneq xD\subseteq I$, so $axD$ is a
nonzero proper submodule of $I$.  Define $f:A\to B$ by sending $1$ to $ax$
in the $I$-summand.  The
domain hypothesis gives $\ker f=0\lesum A$.

If $axD$ were a direct summand of $B$, then
\[
 B/axD\cong D^{r-2}\oplus I/axD
\]
would be projective, and hence $I/axD$ would be projective.  The sequence
\[
 0\longrightarrow axD\longrightarrow I
   \longrightarrow I/axD\longrightarrow0
\]
would split.  This would make $axD$ a nonzero proper direct summand of the
rank-one module $I$, contradicting the indecomposability proved in
\cref{lem:rankone}.  Therefore $P$ is not $\Cfour$.  It cannot be
$\Cfourstar$ or strongly $\Cfourstar$ either.  Finally,
\cref{lem:dedekindsocle} shows that every finitely generated projective $P$ is
semi-weak-CS, proving all the displayed equalities.
\end{proof}

\begin{corollary}[Order units over a Dedekind domain]
Let $D$ be as in \cref{thm:dedekindlocus}.  Every nonzero element of
$\Vmon(D)$ is an order unit.  Among these order units, the $\Cfour$,
$\Cfourstar$, and strongly $\Cfourstar$ conditions hold exactly in rank one;
the semi-weak-CS condition holds on all order units.
\end{corollary}

\begin{proof}
If $P\cong D^{r-1}\oplus I$ has $r\geq2$, it contains $D$ as a summand and is
a generator.  If $r=1$, then $P\cong I$ is invertible and hence a generator.
Thus every nonzero finitely generated projective is a progenerator, so
\cref{prop:orderunit} makes its class an order unit.  The membership statements
are \cref{thm:dedekindlocus}.
\end{proof}

\begin{theorem}[Morita rigidity]\label{thm:dedekindrigidity}
Let $D$ be a Dedekind domain that is not a field, and let $S$ be a ring Morita
equivalent to $D$.  Then the following are equivalent:
\begin{enumerate}[label=(\roman*)]
\item $S$ is a right $\Cfour$-ring;
\item $S$ is a right $\Cfourstar$-ring;
\item $S$ is a strongly right $\Cfourstar$-ring;
\item $S\cong D$.
\end{enumerate}
Moreover, every ring Morita equivalent to $D$ is a right semi-weak-CS ring.
\end{theorem}

\begin{proof}
There is a progenerator $P_D$ such that
$S\cong\End_D(P)$.  By \cref{thm:progenerator}, the regular module $S_S$ has
one of the four properties exactly when $P_D$ has that property.  Since $P$
is nonzero, \cref{thm:dedekindlocus} shows that the first three properties
hold exactly when $P$ has rank one.  In that case $P$ is an invertible ideal
and the multiplication map gives
\[
 \End_D(P)\cong D.
\]
Conversely, $D_D$ is strongly $\Cfourstar$ by
\cref{thm:dedekindlocus}.  The last assertion follows because every
finitely generated projective $D$-module is semi-weak-CS.
\end{proof}

\section{A Morita-robust ring formulation}

The obstruction in \cref{thm:ordercriterion} concerns the choice of one order
unit.  A categorical replacement is obtained by quantifying over the entire
projective monoid.

\begin{definition}\label{def:robust}
For an equivalence-invariant module property $\mathcal Q$, say that $R$ has
$\widehat{\mathcal Q}$ if every finitely generated projective right $R$-module
has $\mathcal Q$.
\end{definition}

\begin{proposition}\label{prop:robustV}
For every ring $R$,
\[
 R\text{ has }\widehat{\mathcal Q}
 \quad\Longleftrightarrow\quad
 \Vmon_{\mathcal Q}(R)=\Vmon(R).
\]
\end{proposition}

\begin{proof}
This is exactly the translation of \cref{def:robust} through
\cref{def:markedV}.
\end{proof}

\begin{theorem}[Projective-universal Morita invariance]\label{thm:robust}
The ring property $\widehat{\mathcal Q}$ is Morita invariant.
\end{theorem}

\begin{proof}
By \cref{thm:markedV}, a Morita equivalence identifies both the ambient
projective monoids and their $\mathcal Q$-marked subsets.  Hence the equality in
\cref{prop:robustV} holds on one side if and only if it holds on the other.
\end{proof}

\begin{proposition}\label{prop:robustmatrix}
If $R$ has $\widehat{\mathcal Q}$, then $M_n(R)$ is a right
$\mathcal Q$-ring for every $n\geq1$.  Conversely, if $\mathcal Q$ is closed
under direct summands and $M_n(R)$ is a right $\mathcal Q$-ring for every
$n\geq1$, then $R$ has $\widehat{\mathcal Q}$.
\end{proposition}

\begin{proof}
The forward implication follows because every $R_R^n$ is finitely generated
projective and \cref{cor:matrix} applies.  For the converse, every finitely
generated projective right $R$-module is a direct summand of $R_R^n$ for some
$n$, while the hypothesis and \cref{cor:matrix} say that each $R_R^n$ has
$\mathcal Q$.
\end{proof}

For $\mathcal Q\in\{\Cfour,\Cfourstar,\text{semi-weak-CS},
\text{strongly }\Cfourstar\}$, \cref{thm:robust} gives a Morita-invariant
ring property by requiring $\mathcal Q$ on every finitely generated
projective module.

\section{All finitely generated modules over Dedekind domains}
\label{sec:mixed-dedekind}

We now pass from the projective locus to arbitrary finitely generated modules.
For a module $M$ over a domain, write
\[
 \tau(M)=\{x\in M:xd=0\text{ for some }0\neq d\in D\}
\]
for its torsion submodule.  If $D$ is Dedekind and $M$ is finitely generated,
then $M/\tau(M)$ is finitely generated, torsion-free, and therefore projective.
Consequently the canonical sequence splits and
\begin{equation}\label{eq:torsion-splitting}
 M\cong P\oplus\tau(M),
 \qquad P\text{ finitely generated projective}.
\end{equation}
The module $\tau(M)$ is annihilated by a nonzero ideal and has finite length.
The issue is not the existence of \eqref{eq:torsion-splitting}, but whether
the $\Cfour$ test can be reduced to its two displayed terms: arbitrary direct
summands of $M$ need not be the chosen summands, and homomorphisms in the test
may cross from a torsion-free part to a torsion part.

\begin{definition}\label{def:c4defect}
A \emph{$\Cfour$ defect} of a module $X$ is a triple $(A,B,f)$ such that
\[
 X=A\oplus B,\qquad f:A\longrightarrow B,\qquad
 \ker f\lesum A,\qquad \im f\not\lesum B.
\]
Write $\operatorname{Def}_{\Cfour}(X)$ for the class of all such defects.
Thus $X$ is $\Cfour$ exactly when
$\operatorname{Def}_{\Cfour}(X)=\varnothing$.
\end{definition}

The following theorem describes the defects under the torsion splitting.

\begin{theorem}[Torsion-defect reconstruction and rank obstruction]
\label{thm:defect-reconstruction}
Let $D$ be a Dedekind domain that is not a field, let $M$ be a finitely
generated $D$-module, and put $T=\tau(M)$ and
$r=\operatorname{rank}_D M$.
\begin{enumerate}[label=(\roman*)]
\item If $r\leq1$, every defect of $T$ lifts explicitly to a defect of $M$,
      and every defect of $M$ reduces explicitly to a defect of $T$.
\item If $r\geq2$, then $M$ has a defect whose domain is free of rank one and
      whose image is a nonzero proper principal submodule of an invertible
      ideal summand of $M/T$.
\end{enumerate}
Consequently,
\[
 \operatorname{Def}_{\Cfour}(M)=\varnothing
 \quad\Longleftrightarrow\quad
 r\leq1\text{ and }
 \operatorname{Def}_{\Cfour}(T)=\varnothing.
\]
\end{theorem}

\begin{proof}
Choose a splitting $M=P\oplus T$.  Assume first that $r\leq1$.  If
$(U,V,g)$ is a defect of $T$, then
\[
 M=U\oplus(P\oplus V),
 \qquad
 \widetilde g:U\longrightarrow P\oplus V,
 \qquad u\longmapsto(0,g(u))
\]
has the same split kernel.  Were $\im g$ a direct summand of $P\oplus V$, a
retraction onto $\im g$, restricted to $V$, would split $\im g$ in $V$.
This contradicts the choice of $(U,V,g)$, so
$(U,P\oplus V,\widetilde g)$ is a defect of $M$.

Conversely, let $(A,B,f)$ be a defect of $M$.  Since
$\operatorname{rank}A+\operatorname{rank}B\leq1$, at least one of $A,B$ is
torsion.  The torsion radical respects the decomposition:
\begin{equation}\label{eq:torsion-components}
 T=\tau(A)\oplus\tau(B).
\end{equation}
If $A$ is torsion, then $A=\tau(A)$ and $f(A)\leq\tau(B)$.  The quotient
$B/\tau(B)$ is projective, so $\tau(B)\lesum B$.  If $\im f$ split in
$\tau(B)$, it would therefore split in $B$.  Hence
$(A,\tau(B),f)$ is a defect of $T$.

Suppose instead that $B$ is torsion.  Put $K=\ker f$ and choose
$A=K\oplus C$.  The restriction $f|_C$ is injective and embeds $C$ in the
torsion module $B$, so $C$ is torsion.  Therefore
\[
 T=C\oplus\bigl(\tau(K)\oplus B\bigr).
\]
Define
\[
 h:C\longrightarrow\tau(K)\oplus B,
 \qquad h(c)=(0,f(c)).
\]
The map $h$ is injective.  If its image were a direct summand of
$\tau(K)\oplus B$, a retraction restricted to $B$ would split $\im f$ in
$B$, again contradicting that $(A,B,f)$ is a defect.  Thus
$(C,\tau(K)\oplus B,h)$ is a defect of $T$.  This proves (i), including both
explicit witness transformations.

Now let $r\geq2$.  By Steinitz decomposition,
\[
 M\cong D\oplus\bigl(D^{r-2}\oplus I\oplus T\bigr)=A\oplus B
\]
for an invertible ideal $I$.  Choose a nonzero nonunit $a\in D$ and
$0\neq x\in I$, and define $f:A\to B$ by sending $1$ to $ax$ in the
$I$-coordinate.  This map is injective.  Moreover,
$0\neq axD\subseteq aI\subsetneq I$; the last inclusion is strict because
$I$ is invertible and $a$ is not a unit.  If $axD$ split in $B$, a retraction
restricted to $I$ would make $axD$ a nonzero proper direct summand of the
rank-one module $I$, contradicting \cref{lem:rankone}.  Hence
$(A,B,f)$ is the asserted defect, proving (ii).

The final equivalence follows from (i), (ii), and
\cref{def:c4defect}.
\end{proof}

\begin{corollary}[Torsion-radical reduction for $\Cfour$]
\label{thm:torsion-reduction-c4}
Under the hypotheses of \cref{thm:defect-reconstruction},
\[
 M\text{ is }\Cfour
 \quad\Longleftrightarrow\quad
 \operatorname{rank}_D M\leq1
 \text{ and }\tau(M)\text{ is }\Cfour.
\]
\end{corollary}

\begin{proof}
Apply the final equivalence of \cref{thm:defect-reconstruction} and the
characterization in \cref{def:c4defect}.
\end{proof}

The hereditary condition has a corresponding reduction.

\begin{theorem}[Torsion-radical reduction for $\Cfourstar$]
\label{thm:torsion-reduction-c4star}
Under the hypotheses of \cref{thm:torsion-reduction-c4},
\[
 M\text{ is }\Cfourstar
 \quad\Longleftrightarrow\quad
 \operatorname{rank}_D M\leq1
 \text{ and }\tau(M)\text{ is }\Cfourstar.
\]
\end{theorem}

\begin{proof}
If $M$ is $\Cfourstar$, then it is $\Cfour$, so
\cref{thm:torsion-reduction-c4} gives
$\operatorname{rank}_D M\leq1$.  Every submodule of $\tau(M)$ is also a
submodule of $M$, and is therefore $\Cfour$; hence $\tau(M)$ is
$\Cfourstar$.

Conversely, assume the two displayed conditions and let $N\leq M$.  Since
$D$ is Noetherian, $N$ is finitely generated; moreover,
\[
 \operatorname{rank}_D N\leq1,
 \qquad \tau(N)=N\cap\tau(M)\leq\tau(M).
\]
The $\Cfourstar$ property of $\tau(M)$ implies that $\tau(N)$ is $\Cfour$.
Applying \cref{thm:torsion-reduction-c4} to $N$ shows that $N$ is $\Cfour$.
Thus every submodule of $M$ is $\Cfour$, as required.
\end{proof}

The chain part of the strong condition becomes automatic under finite
generation.

\begin{lemma}[Noetherian modules]\label{lem:noetherian-swcs}
Every Noetherian module is semi-weak-CS.
\end{lemma}

\begin{proof}
Let $\mathcal C$ be a chain in $\mathcal S(M)$.  The two coordinate unions
$X_{\mathcal C}$ and $Y_{\mathcal C}$ are finitely generated.  A finite set of
their generators is contained in the coordinates of finitely many members of
the chain.  Because the chain is linearly ordered, one member lies above all
those finitely many members.  Its coordinates are therefore exactly
$X_{\mathcal C}$ and $Y_{\mathcal C}$.  These are direct summands of $M$ and
are essential in themselves, which verifies \cref{def:swcs}.
\end{proof}

We can now state the criteria in invariant-factor form.  The finite-length
torsion criterion is
\citep[Theorem~6.4]{Gokavarapu2026SmithOrbit}; the passage to arbitrary
finitely generated modules uses
\cref{thm:torsion-reduction-c4,thm:torsion-reduction-c4star}.

\begin{theorem}[Invariant-factor criteria over Dedekind domains]
\label{thm:complete-dedekind-classification}
Let $D$ be a Dedekind domain that is not a field and let $M$ be a finitely
generated $D$-module.  Put $T=\tau(M)$ and
$r=\operatorname{rank}_D M$.  For every nonzero prime ideal $\mathfrak p$ in
the support of $T$, write
\[
 T_{\mathfrak p}
 \cong
 \bigoplus_{j=1}^{m_{\mathfrak p}}
 D/\mathfrak p^{\lambda_{\mathfrak p j}},
 \qquad
 1\leq\lambda_{\mathfrak p1}\leq\cdots\leq
 \lambda_{\mathfrak p m_{\mathfrak p}}.
\]
Then:
\begin{enumerate}[label=(\roman*)]
\item $M$ is $\Cfour$ if and only if $r\leq1$ and, for every
      $\mathfrak p$, all the exponents
      $\lambda_{\mathfrak p j}$ are equal; equivalently, every primary
      component $T_{\mathfrak p}$ is homocyclic.
\item $M$ is $\Cfourstar$ if and only if $r\leq1$ and, for every
      $\mathfrak p$, either $m_{\mathfrak p}=1$ or
      $\lambda_{\mathfrak p1}=\cdots=
      \lambda_{\mathfrak p m_{\mathfrak p}}=1$; equivalently, every primary
      component is cyclic or semisimple.
\item Every finitely generated $D$-module is semi-weak-CS.
\item $M$ is strongly $\Cfourstar$ if and only if it is $\Cfourstar$, and
      hence exactly under the criterion in \textup{(ii)}.
\end{enumerate}
\end{theorem}

\begin{proof}
The finite-length torsion theorem
\citet[Theorem~6.4]{Gokavarapu2026SmithOrbit} says that $T$ is $\Cfour$
exactly when all exponents in each primary component are equal, and that $T$
is $\Cfourstar$ exactly when each primary component is cyclic or semisimple.
Combine these two criteria with
\cref{thm:torsion-reduction-c4,thm:torsion-reduction-c4star} to obtain
\textup{(i)} and \textup{(ii)}.  Since $D$ is Noetherian, $M$ is Noetherian,
so \cref{lem:noetherian-swcs} proves \textup{(iii)}.  Assertion \textup{(iv)}
follows from \textup{(ii)}, \textup{(iii)}, and \cref{def:swcs}.
\end{proof}

The same criteria determine when a direct sum has the relevant property.
For a nonzero homocyclic $\mathfrak p$-primary module $X$, let
$e_{\mathfrak p}(X)$ denote its common exponent.

\begin{theorem}[Direct-sum criterion]
\label{thm:direct-sum-compatibility}
Let $M$ and $N$ be finitely generated modules over a nonfield Dedekind domain
$D$, and write $T_M=\tau(M)$ and $T_N=\tau(N)$.
\begin{enumerate}[label=(\roman*)]
\item The module $M\oplus N$ is $\Cfour$ if and only if
      \[
       \operatorname{rank}M+\operatorname{rank}N\leq1,
      \]
      every primary component of both $T_M$ and $T_N$ is homocyclic, and,
      whenever $(T_M)_{\mathfrak p}$ and $(T_N)_{\mathfrak p}$ are both
      nonzero,
      \[
       e_{\mathfrak p}((T_M)_{\mathfrak p})
       =e_{\mathfrak p}((T_N)_{\mathfrak p}).
      \]
\item The module $M\oplus N$ is $\Cfourstar$ if and only if the same total
      rank bound holds, every primary component of $T_M$ and $T_N$ is cyclic
      or semisimple, and, on every prime in their common support, both
      primary components are semisimple.
\item The module $M\oplus N$ is always semi-weak-CS, and it is strongly
      $\Cfourstar$ exactly under the conditions in \textup{(ii)}.
\end{enumerate}
\end{theorem}

\begin{proof}
The torsion radical and rank are additive on finite direct sums:
\[
 \tau(M\oplus N)=T_M\oplus T_N,
 \qquad
 \operatorname{rank}(M\oplus N)
 =\operatorname{rank}M+\operatorname{rank}N.
\]
At a fixed prime $\mathfrak p$, the invariant-factor multiset of
$(T_M\oplus T_N)_{\mathfrak p}$ is the union of the two invariant-factor
multisets.  This union has one common exponent exactly when both nonzero
components are homocyclic with the same exponent.  Applying
\cref{thm:complete-dedekind-classification}(i) proves (i).

For (ii), suppose both primary components are nonzero.  Their direct sum is
semisimple precisely when both are semisimple.  It cannot be cyclic unless
one component is zero: the minimal number of generators over the local ring
$D_{\mathfrak p}$ is additive for nonzero finite-length direct summands.
Hence the direct sum is cyclic or semisimple exactly when either only one
component is nonzero and is cyclic or semisimple, or both are nonzero and
semisimple.  Apply
\cref{thm:complete-dedekind-classification}(ii).  Finally, (iii) follows from
parts (iii) and (iv) of that classification theorem.
\end{proof}

\begin{corollary}[Modules with projective and torsion parts]\label{cor:mixed-examples}
Let $I$ be an invertible ideal of $D$ and let $0\neq T$ be a finite-length
torsion module.  Then $I\oplus T$ is $\Cfour$ exactly when every primary
component of $T$ is homocyclic.  It is $\Cfourstar$, equivalently strongly
$\Cfourstar$, exactly when every primary component of $T$ is cyclic or
semisimple.
\end{corollary}

\begin{proof}
The module $I\oplus T$ has rank one and torsion submodule $T$.  Apply
\cref{thm:complete-dedekind-classification}.
\end{proof}

\section{Conclusion}

We proved that $\Cfourstar$, semi-weak-CS, and strongly $\Cfourstar$ are
preserved and reflected by equivalences of module categories.  The
semi-weak-CS proof uses the joins of the compatible chains in the defining
subobject lattice.

For a module property $\mathcal Q$ preserved and reflected by equivalences,
the marked pair
\[
 (\Vmon(R),\Vmon_{\mathcal Q}(R))
\]
is transported by Morita equivalence.  The regular-module ring property is
Morita invariant if and only if the marking is constant on the order units.
The progenerator, matrix, and full-corner formulas are consequences of this
comparison, and the matrix profiles satisfy the scaling formula in
\cref{thm:profilescaling}.  Over a nondivision right Ore domain, the profiles
of $\Cfour$, $\Cfourstar$, and strongly $\Cfourstar$ are $\{1\}$, while
the semi-weak-CS profile is $\mathbb N_{\geq1}$.

For a nonfield Dedekind domain, the $\Cfour$, $\Cfourstar$, and strongly
$\Cfourstar$ marked loci are the rank-zero and rank-one projective classes;
the semi-weak-CS marked locus is $\Vmon(D)$.  It follows that a ring in the
Morita class of $D$ is right $\Cfour$, right $\Cfourstar$, or strongly right
$\Cfourstar$ precisely when it is isomorphic to $D$.  Every ring in that
Morita class is right semi-weak-CS.

For a finitely generated $D$-module $M$, the defect theorem proves
\[
 M\text{ is }\Cfour
 \quad\Longleftrightarrow\quad
 \operatorname{rank}_D M\leq1
 \text{ and }\tau(M)\text{ is }\Cfour,
\]
and the analogous equivalence for $\Cfourstar$.  Combining these reductions
with the cited finite-length torsion theorem gives the invariant-factor
criteria in \cref{thm:complete-dedekind-classification}.  The direct-sum
criterion adds the conditions on common primary support stated in
\cref{thm:direct-sum-compatibility}.  Every finitely generated $D$-module is
semi-weak-CS, so the $\Cfourstar$ criterion also characterizes strongly
$\Cfourstar$ modules in this setting.

Further work can ask for comparable marked-locus and torsion-radical
calculations over hereditary Noetherian prime, semiperfect, or serial rings,
where the projective and torsion classifications are subtler than the
rank--Steinitz and primary-decomposition descriptions used here.

\section*{Author contributions}

Chandrasekhar Gokavarapu (lead and corresponding author): conceptualization
and methodology; writing---original draft; project administration.  Sekhar
Babu Gosala: formal analysis and proofs; writing---review and editing.
Sudhakar Gadde: formal analysis and proofs; writing---review and editing.
Rajasekhar Yedla: investigation and verification; validation;
writing---review and editing.  Durga Ratna Sai Ryali: validation; notation
curation; writing---review and editing.

\section*{Preprint statement}

Version 1 of this article is available as a preprint at
\href{https://arxiv.org/abs/2604.16326v1}{arXiv:2604.16326v1}.  The present
article is the revised second version and will replace Version 1 on arXiv
shortly.

\section*{Disclosure statement}

The authors report no potential conflict of interest.

\section*{Funding}

No external funding was received for this work.

\section*{Data availability statement}

No datasets were generated or analysed because this article is purely
theoretical.

\def\doi#1#2{}
\bibliographystyle{plainnat}
\bibliography{refs}

\end{document}